\pdfoutput=1
\documentclass[oneside, a4paper, 12pt]{amsart}
\usepackage{setspace}
\usepackage[inner=1in,outer=1in,tmargin=1in,bmargin=1in]{geometry}
\usepackage{mathtools}
\usepackage{amssymb}
\usepackage[utf8]{inputenc}
\usepackage[T1]{fontenc}
\usepackage[english]{babel}
\usepackage{parskip}
\usepackage{tensor}
\usepackage{hyperref}
\usepackage{amsmath}
\usepackage{afterpage}
\usepackage{tikz}
\usetikzlibrary{3d,calc,arrows.meta,patterns,decorations.pathmorphing,decorations.pathreplacing}
\usepackage{ae,aecompl}
\usepackage{graphicx}
\usepackage[capposition=bottom]{floatrow}
\usepackage{empheq}
\usepackage{amsthm}
\usepackage{lipsum} 
\usepackage{cite}
\usepackage{cases}
\usepackage{hyperref}
\usepackage{pgfplots}
\usepackage{enumitem}
\usepackage{pdfpages}
\usepackage{slashed}
\usepackage[colorinlistoftodos]{todonotes}

\allowdisplaybreaks

\hypersetup{
	colorlinks,
	linkcolor={red!80!black},
	citecolor={blue!80!black},
	urlcolor={blue!80!black},
	linktocpage=page
}

\usetikzlibrary{decorations.pathmorphing,calc}


\providecommand{\abs}[1]{\lvert#1\rvert}

\newcommand{\munud}{_{\mu\nu}}

\newcommand{\pt}{\partial}

\newcommand{\be}{\begin{equation}}
	\newcommand{\ee}{\end{equation}}
\newcommand{\ba}{\begin{eqnarray}}
	\newcommand{\ea}{\end{eqnarray}}
\newcommand{\baa}{\begin{array}}
	\newcommand{\eaa}{\end{array}}

\newcommand{\ijd}{_{ij}}

\newcommand{\rmi}[1]{{\mbox{\scriptsize #1}}}  

\newcommand{\pminus}[1]{^{-#1}}

\makeatletter
\newcommand\incircbin
{%
  \mathpalette\@incircbin
}
\newcommand\@incircbin[2]
{%
  \mathbin%
  {%
    \ooalign{\hidewidth$#1#2$\hidewidth\crcr$#1\bigcirc$}%
  }%
}

\makeatother

\theoremstyle{definition}
\newtheorem{proposition}{Proposition}[section]

\theoremstyle{definition}
\newtheorem{theorem}{Theorem}[section]

\theoremstyle{definition}
\newtheorem{lemma}{Lemma}[section]

\theoremstyle{definition}

\theoremstyle{definition}

\theoremstyle{definition}
\newtheorem{question}{Question}[section]

\theoremstyle{remark}
\newtheorem{remark}{Remark}
\numberwithin{equation}{section}

\newcommand{\R}{\mathbb{R}}

\newcommand{\Cinf}{C^\infty}
\newcommand{\Cinfo}{C^\infty_0}

\def\p{\partial}
\DeclareMathOperator{\supp}{supp}

\begin{document}

\begin{abstract}
We show that two non-isometric, smooth, globally hyperbolic Lorentzian metrics can have the same hyperbolic Dirichlet-to-Neumann map on an infinite cylinder with timelike boundary. 
\end{abstract}

\title{Counterexamples to the Lorentzian Calderón problem}

\author{Lauri Oksanen, Miika Sarkkinen}

\address{Department of Mathematics and Statistics, University of Helsinki, PO BOX 68, 00014, Helsinki, Finland}
\email {lauri.oksanen@helsinki.fi}

\address{Department of Mathematics and Statistics, University of Jyväskylä, Jyväskylä, Finland
\newline Department of Mathematics and Statistics, University of Helsinki, PO BOX 68, 00014, Helsinki, Finland}
\email {miika.j.sarkkinen@jyu.fi, miika.sarkkinen@helsinki.fi}

\maketitle

\section{Introduction}

Positive results on the Lorentzian Calder\'on problem have been obtained recently \cite{stefanov2018, alexakis2022, alexakis2025, oksanen2024rigiditylorentziancalderonproblem, 2505.13676, 2505.15189, 2512.19601}. These results impose rather stringent geometric conditions. Partly motivating such conditions, we show that having data on an infinite cylinder with timelike boundary is not enough for the problem to be uniquely solvable. Our counterexamples are similar to earlier examples in the context of inverse scattering for moving obstacles \cite{Stefanov1991} in the sense that they are based on finite speed of propagation. We refer to \cite{liimatainen2022} for examples with data on only a part of a cylinder, and to \cite{eskin2019} for non-smooth, cloaking type examples. Cloaking has been extensively studied in the context of the classical (Riemannian) version of Calder\'on problem, see the review \cite{greenleaf2009a}. We mention also two studies of spacetime cloaking in physics literature \cite{McCall2010, Fridman2012}.

\subsection{Statement of results} Let $\mathcal M$ be a noncompact, connected smooth manifold of $1+n, n \geq 1$. Let $g$ be a time-orientable Lorentzian metric on $\mathcal M$.  We call $(\mathcal M,g)$ with such properties a \emph{spacetime}. We assume that the spacetime is globally hyperbolic, which means that \cite{Bernal_2007} 
\begin{enumerate}
    \item $\mathcal M$ contains no closed causal curves, and
    \item the causal diamonds $J_g^-(p) \cap J_g^+(q)$ are compact for any $p,q \in \mathcal M$.
\end{enumerate}
We note that for dimension $1+n \geq 3$, which is the dimension we are interested in on physical grounds, the first condition can be dropped from the definition \cite{hounnonkpe2019}. Global hyperbolicity can also be equivalently characterized by saying that $\mathcal M$ contains a Cauchy hypersurface, that is, a subset $\Sigma \subset \mathcal M$ that each inextendible timelike curve crosses exactly once \cite[Thm. 11]{geroch1970}. Moreover, results due to Bernal and Sánchez \cite{bernal2003, bernal2005} guarantee that a globally hyperbolic manifold can be foliated by smooth spacelike Cauchy hypersurfaces given by the level sets of a temporal function $\tau \in \Cinf(\mathcal M)$, that is, a function whose differential $d\tau$ is timelike. Such a function is called a \emph{Cauchy temporal function}. For any inextendible causal curve $\gamma: \R \to \mathcal M$ it holds that $\abs{\tau(\gamma(s))}\to \infty$ as $s \to \pm \infty$. Moreover, we have a global representation for the metric $g$ in terms of $\tau$:
\begin{theorem}\label{thm:bernal}
Let $(\mathcal M,g)$ be a globally hyperbolic spacetime with a Cauchy temporal function $\tau$. Then $(\mathcal M,g)$ is isometric to the smooth product manifold $\R\times \Sigma$ with the metric
\[
-\kappa d\tau^2 + g_\tau,
\]
where $\kappa \in \Cinf(\R\times\Sigma)$ is strictly positive, and $g_\tau$ is a Riemannian metric on $\Sigma$ that depends on $\tau$.
\end{theorem}

Globally hyperbolic spacetimes form a natural setting for inverse problems for the wave equation, as global hyperbolicity implies that the Cauchy problem for the wave equation is well-posed and solutions thereof have a finite propagation speed \cite[Thms. 3.2.11, 3.2.12]{bar2007}. The Lorentzian Calderón problem is an inverse boundary value problem for hyperbolic partial differential equations on Lorentzian manifolds, where one seeks to determine the underlying geometry of a spacetime from measurements at the boundary. This problem can be viewed as a Lorentzian analogue of the classical Calderón problem in the Riemannian setting, but exhibits fundamentally different features due to the causal structure of the wave equation. 

Let us now give a precise statement of the problem. We say that $M\subset \mathcal M$ is a {\em cylinder} if it is a submanifold of the same dimension with $\mathcal M$ and has a smooth timelike boundary $\pt M$. Consider the wave equation on such $M$
\begin{equation}
    \Box_g u = 0, \quad u\rvert_{\pt M} = \varphi, \quad u\rvert_{\tau\ll 0} = 0,
\end{equation}
where $\varphi \in \Cinfo(\pt M)$, and $\Box_g$ is the Laplace-Beltrami operator defined by
\[
\Box_g u = -\abs{g}^{-1/2} \pt_j (\abs{g}^{1/2}g^{jk}\pt_k u).
\]
Here the Einstein summation convention was implied. Then we define the hyperbolic Dirichlet-to-Neumann (DN) map
\[
\Lambda_g^\textnormal{Hyp}: \Cinfo(\pt M) \to \Cinf(\pt M), \quad \Lambda_g^\textnormal{Hyp}\varphi = \pt_\nu u\rvert_{\pt M},
\]
where $\nu$ is the unit inwards pointing normal vector to $\pt M$. In the Lorentzian Calderón problem, the data is given by the hyperbolic DN map, and the inverse problem is to recover the metric $g$, up to boundary-fixing diffeomorphisms.
In this paper, we give negative answers to the following uniqueness question for the Lorentzian Calderón problem.
\begin{question}\label{question:uniqueness}
Let $g$ and $g'$ be globally hyperblic Lorentzian metrics on $\mathcal M$, and let $M$ be a cylinder with respect to both the metrics. Suppose that $\Lambda_g^\textnormal{Hyp} = \Lambda_{g'}^\textnormal{Hyp}$.  Are $g$ and $g'$ isometric by a boundary-fixing diffeomorphism $F: M \to M$, $F\rvert_{\pt M} = \textnormal{id}$?
\end{question}

We will provide examples of smooth Lorentzian manifolds $(M,g)$ obtained by choosing a cylinder inside a globally hyperbolic spacetime $(\mathcal M,g)$, such that uniqueness in the Lorentzian Calderón problem fails therein. These include a cylinder in Minkowski spacetime given by a hyperboloid, Schwarzschild black hole and white hole spacetimes, and certain type of cosmological spacetimes. We show that these spacetimes can be locally perturbed in such a fashion that the DN map does not detect the resulting change in the geometry.

The counterexamples are based on the existence of a nonempty open set $U \subset M$ that is unreachable from $\p M$, meaning 
    \begin{align}\label{unreachable}
J_g^-(U) \cap \pt M = \varnothing 
\quad \text{or} \quad 
J_g^+(U) \cap \pt M = \varnothing.
    \end{align}
To state this common basis as a proposition, consider first the problem
\begin{equation}
    \Box_g u = f, \quad \supp u \subset J_g^+(\supp f),
\end{equation}
where $f \in \Cinfo(\mathcal{M}\backslash M)$. We may then define the \emph{source-to-solution map}
\[
L_g: \Cinfo(\mathcal{M}\backslash M) \to \Cinf(\mathcal M\backslash M), \quad f \mapsto u\rvert_{\mathcal M\backslash M}.
\] 
\begin{proposition}\label{prop:nonuniqueness}
Let $(\mathcal M, g)$ be a globally hyperbolic spacetime with a Cauchy temporal function $\tau$, and let $M \subset \mathcal M$ be a cylinder. 
Suppose there exists a nonempty open set $U \subset M$ satisfying \eqref{unreachable}. Then there exists a globally hyperbolic metric $g'$ on $\mathcal M$ such that $\tau$ is a Cauchy temporal function with respect to $g'$, $M$ is a cylinder with respect to $g'$, and that $L_g = L_{g'}$ but $g$ is not isometric to $g'$.
\end{proposition}
Moreover, to obtain a negative answer to Question~\ref{question:uniqueness}, we combine Proposition~\ref{prop:nonuniqueness} with
\begin{proposition}\label{lem_data_implication}
Let $g$ and $g'$ be globally hyperblic Lorentzian metrics on $\mathcal M$, and let $\tau$ be a Cauchy temporal function and $M$ a cylinder with respect to both the metrics. Suppose $L_g = L_{g'}$. Then $\Lambda_g^\textnormal{Hyp} = \Lambda_{g'}^\textnormal{Hyp}$.
\end{proposition}
We will prove Propositions~\ref{prop:nonuniqueness} and~\ref{lem_data_implication} in Section~\ref{section:proofs}.

\subsection{Some basic notions of the Lorentzian causality theory}\label{section:causality_theory}
Let $(\mathcal{M}, g)$ be a Lorentzian manifold with metric of signature $(-,+,\dots,+)$. A vector $v \in T_p\mathcal{M}$ is called \emph{timelike}, \emph{null} (or \emph{lightlike}), or \emph{spacelike} if $g(v,v) < 0$, $g(v,v) = 0$ with $v \neq 0$, or $g(v,v) \geq 0$ with equality only if $v=0$, respectively. A vector is called \emph{causal} if it is either timelike or null.

The manifold $(\mathcal{M}, g)$ is said to be \emph{time-oriented} if there exists a smooth timelike vector field $T$ on $\mathcal{M}$. In this case, a causal vector $v \in T_p\mathcal{M}$ is said to be \emph{future-pointing} if $g(v,T) < 0$, and \emph{past-pointing} if $g(v,T) > 0$.

A piecewise smooth curve $\gamma : I \to \mathcal{M}$ is said to be \emph{timelike}, \emph{null}, or \emph{causal} if its tangent vector $\dot{\gamma}(t)$ is timelike, null, or causal for all $t \in I$. It is \emph{future-directed} (resp.\ \emph{past-directed}) if $\dot{\gamma}(t)$ is future-pointing (resp.\ past-pointing) for all $t$.

We write $p \ll q$ for $p, q \in \mathcal M$ if there exists a future-directed timelike curve from $p$ to $q$, $p < q$ if there exists a future-directed causal curve from $p$ to $q$, and $p \le q$ if $p < q$ or $p = q$. The \emph{chronological} and \emph{causal future} of $p$ are
\[
I_g^+(p) := 
\{q \in \mathcal{M} \;:\; p \ll q\}, 
\quad 
J_g^+(p) := \{ q \in \mathcal{M} \;:\; p \le q \},
\]
and the \emph{chronological} $I^-(p)$ and \emph{causal past} $J_g^-(p)$ are defined analogously. Further, for a set $A \subset \mathcal M$, 
    \begin{align*}
I_g^\pm(A) = \bigcup_{p \in A} I_g^+(p),
\quad
J_g^\pm(A) = \bigcup_{p \in A} J_g^+(p).
    \end{align*}
One has $I^\pm(A) \subset J^\pm(A)$, and $I^\pm(A)$ are open sets for any $A \subset \mathcal M$, while $J^\pm(p)$ need not be closed in general \cite[Lem. 14.2]{oneill1983}. If $(\mathcal M, g)$ is globally hyperbolic, then $J^\pm(K)$ are closed for compact $K \subset \mathcal M$
\cite{hounnonkpe2019}.

A causal (or timelike) curve $\gamma : I \to \mathcal{M}$ is said to be \emph{future inextendible} if it has no future endpoint, i.e., there is no point $p \in \mathcal{M}$ such that $\gamma(t) \to p$ as $t \to \sup I$. Similarly, $\gamma$ is \emph{past inextendible} if it has no past endpoint, i.e., there is no point $p \in \mathcal{M}$ such that $\gamma(t) \to p$ as $t \to \inf I$. The curve $\gamma$ is called \emph{inextendible} if it is both future and past inextendible.
Equivalently, a causal curve is inextendible if it cannot be extended to a strictly larger interval, which in this formulation is expressed by the absence of endpoints in $\mathcal{M}$.

\subsection*{Acknowledgements}
LO and MS were supported by the European Research Council of the European Union, grant 101086697 (LoCal), and the Reseach Council of Finland, grants 347715, 353096 (Centre of Excellence of Inverse Modelling and Imaging) and 359182 (Flagship of Advanced Mathematics for Sensing Imaging and Modelling). MS was also supported by the Research Council of Finland project \#2100006347. Views and opinions expressed are those of the authors only and do not necessarily reflect those of the European Union or the other funding organizations.

\section{Counterexamples}
\subsection{Minkowski hyperboloid}
The simplest Lorentzian manifold is the Minkowski space of special relativity, that is, the manifold $\R^{1+n}$ equipped with the metric
\[
g = -dt^2 + \delta\ijd dx^i dx^j,
\]
where $(t,x)$ are the natural coordinates on $\R^{1+n}$. The Minkowski space is the paradigm case of a globally hyperbolic manifold where the level surfaces of $t$ give a smooth foliation by Cauchy hypersurfaces. Up to an affine reparametrization, a future-directed null geodesic of $g$ can be written in the form
\begin{equation}\label{eq_null_geodesic}
\gamma(s) = (t(s),x(s)) = (t_0 + s, x_0 + sv), \quad (t_0,x_0) \in \R^{1+n},
\end{equation}
where $v \in S^{n-1} \subset \R^n$. Now take a domain $M \subset \R^{1+n}$ defined by
\[
M = \{ (t,x) \in \R^{1+n}: \abs{x}^2 - a \abs x - t^2 \leq 0 \}, \quad  a > 0,
\]
and compute the differential of the function $f(t,x) = \abs{x}^2 - a \abs x - t^2$:
\[
df = - 2t dt + 2 \delta\ijd x^j dx^i - a \abs{x}\pminus1 \delta\ijd x^j dx^i. 
\]
Observe that the boundary $\pt M = \{f = 0\}$ is timelike; indeed, we have
\[
g(df,df)\rvert_{\pt M} =  4 (-t^2 + \abs{x}^2 - a \abs x)\rvert_{\pt M} + a^2 = a^2 > 0.
\]
Consider then the the diamond-shaped domain
\begin{equation}\label{eq:diamond}
    U = \{ (t,x) \in \R^{1+n}: \abs t + \abs x < a /2 \}.
\end{equation}

A null geodesic $\gamma$ given by \eqref{eq_null_geodesic} with $(t_0, x_0) \in U$ stays in the interior of $M$. Indeed, 
    \begin{align*}
\abs{x(s)} 
\le 
|x_0 - t_0 v| + |t_0 v + sv| 
= 
|x_0 - t_0 v| + |t(s)|
< a / 2 + |t(s)|,
    \end{align*}
and
    \begin{align*}
\abs{x(s)}^2 - a \abs{x(s)} - |t(s)|^2
< 
\abs{x(s)}^2 - a \abs{x(s)} - (a / 2 - \abs{x(s)})^2
= 
- a^2/4 < 0.
    \end{align*}
Consequently, $U \cap (J^-(\pt M) \cup J^+(\pt M)) = \varnothing $, and by Propositions~\ref{prop:nonuniqueness} and~\ref{lem_data_implication} we have a counterexample to Question~\ref{question:uniqueness}. 
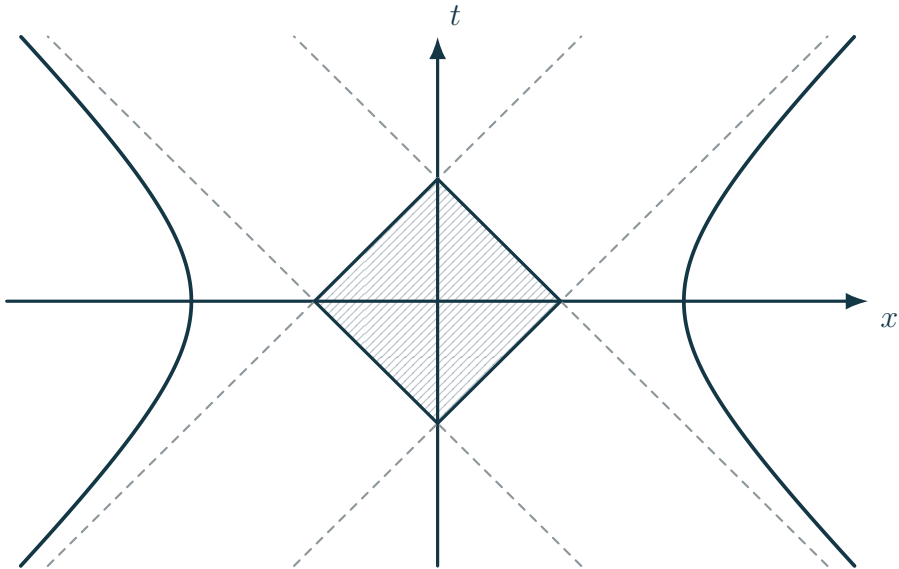
\begin{figure}
\begin{tikzpicture}[scale=1.35, line cap=round, line join=round]

\definecolor{ink}{RGB}{20,55,70}
\definecolor{soft}{RGB}{140,150,155}

\def\a{-2.4}      
\def\xmax{4.2}
\def\tmax{2.6}
\def\tmin{-2.6}

\pgfmathsetmacro{\d}{-0.5*\a}

\draw[ink, line width=1.2pt, -{Latex[length=3mm]}] (0,\tmin) -- (0,\tmax) node[above right] {$t$};
\draw[ink, line width=1.2pt, -{Latex[length=3mm]}] (-\xmax,0) -- (\xmax,0) node[below right] {$x$};

\draw[soft, dashed, line width=0.8pt]
  plot[domain=\tmin:\tmax] ({\x - \d}, {\x});

\draw[soft, dashed, line width=0.8pt]
  plot[domain=\tmin:\tmax] ({\d - \x}, {\x});

\draw[soft, dashed, line width=0.8pt]
  plot[domain=\tmin:\tmax] ({\x + \d}, {\x});

\draw[soft, dashed, line width=0.8pt]
  plot[domain=\tmin:\tmax] ({-\d - \x}, {\x});

\draw[ink, line width=1.2pt]
  (0,\d) -- (\d,0) -- (0,-\d) -- (-\d,0) -- cycle;
\fill[pattern=north east lines, pattern color=ink, opacity=0.55]
  (0,\d) -- (\d,0) -- (0,-\d) -- (-\d,0) -- cycle;

\draw[ink, line width=1.4pt]
  plot[smooth, samples=240, domain=\tmin:\tmax]
    ({(-\a + sqrt(\a*\a + 4*\x*\x))/2}, {\x});

\draw[ink, line width=1.4pt]
  plot[smooth, samples=240, domain=\tmin:\tmax]
    ({(\a - sqrt(\a*\a + 4*\x*\x))/2}, {\x});

\end{tikzpicture}
\caption{A cross-section of the Minkowski domain with a timelike hyperboloidal boundary. The shaded region is the domain $U$ in~\eqref{eq:diamond}. The hyperboloid asymptotes to the boundaries of the causal future and past of $U$ but never touches them.}
    \label{fig:minkowski}
\end{figure}

\subsection{Black and white hole spacetimes}
The unique spherically symmetric solution to the vacuum Einstein equation is the Schwarzschild metric, expressed in coordinates as
\begin{equation}\label{eq:schwarzschild}
g = -(1- r_S/r)dt^2 + (1-r_S/r)\pminus 1 dr^2 + r^2 g_{S^2}, \quad t\in \R, \,r \in (0,r_S)\cup (r_S, \infty),
\end{equation}
where $r_S>0$ is a constant called the Schwarzschild radius and $g_{S^2}$ is the unit two-sphere metric. The Schwarzschild spacetime has a curvature singularity at $r = 0$, whereas the singularity at $r= r_S$ is merely a coordinate artefact. Introducing the null coordinates
\begin{align}
    u &= -\left( \frac{r}{r_S} -1 \right)^{1/2} e^{(r-t)/2 r_S}, \\
    v &= \left( \frac{r}{r_S} -1 \right)^{1/2} e^{(r+t)/2 r_S},
\end{align}
the metric becomes
\begin{equation}
    g = - \frac{4 r_S^3}{r} e^{-r/r_S}dudv + r^2 g_{S^2}.
\end{equation}
It is convenient to further transform to a coordinate system where we have only timelike and spacelike coordinates, namely the Kruskal coordinates $T,R$ defined by $T = (v + u)/2$ and $R = (v-u)/2$.
The metric now reads
\[
g = \frac{4 r_S^3}{r} e^{-r/r_S}(-dT^2 + dR^2) + r^2 g_{S^2}.
\]
The coordinate $r$ is determined implicitly from
\[
T^2 - R^2 = \left( 1 - \frac{r}{r_S} \right)e^{r/r_S},
\]
and coordinates $T,R$ have ranges $R \in (-\infty, \infty), T^2 < R^2 + 1$. The Kruskal coordinates cover the whole maximally extended Schwarzschild spacetime, also known as the Kruskal spacetime $\mathcal K$ \cite[Cor. 13.37]{oneill1983}. The spacetime $\mathcal K$ contains both a black hole region $\mathcal B =\{r<r_S, T>0\}$ where no causal curve can escape (by \cite[Prop. 13.30]{oneill1983}), and a white hole region $\mathcal W =\{r<r_S, T<0\}$ that no causal curve can enter (due to time reversal symmetry of the metric). The Kruskal spacetime is globally hyperbolic. Indeed, the level set $\{T = 0 \}$ is a Cauchy hypersurface \cite[pp. 418-419]{oneill1983}.

Consider the submanifold $M=\{p \in \mathcal K: r(p) < r_0 \}$ for some $r_0 > r_S$. From the representation~\eqref{eq:schwarzschild} of the metric, it is clear that $g(dr,dr)\rvert_{r=r_0}>0$ so that the boundary $\pt M = \{r=r_0\}$ is timelike. It follows from \cite[Prop. 13.30]{oneill1983} that $J^+(\mathcal B)$ does not intersect with any subset of $\{r > r_S\}\subset \mathcal K$, in particular with $\pt M$. 
On the other hand, the white hole region $\mathcal W$ satisfies $J^-(\mathcal W) \cap \pt M = \varnothing$.
Hence, Proposition~\ref{prop:nonuniqueness} can be applied to $\mathcal B$ or $\mathcal W$ (or any $U$ contained in one of them) to obtain, together with Proposition \ref{lem_data_implication}, a negative answer to Question~\ref{question:uniqueness}.

\begin{figure}
\begin{tikzpicture}[scale=1.15, line cap=round, line join=round]

\definecolor{ink}{RGB}{20,55,70}
\definecolor{soft}{RGB}{140,150,155}

\tikzset{
  axis/.style={ink, line width=1.2pt, -{Latex[length=3mm]}},
  curve/.style={ink, line width=1.05pt, opacity=0.75},
  horizon/.style={soft, dashed, line width=0.9pt},
  sing/.style={ink, line width=1.25pt,
               decorate, decoration={zigzag, segment length=4.2pt, amplitude=1.0pt}},
}

\def\Rx{4.6}
\def\Ty{3.2}
\def\b{2.35}

\def\nInterior{2}
\def\nExterior{2}

\begin{scope}
  \clip (-\Rx,-\Ty) rectangle (\Rx,\Ty);

  \draw[horizon] (-\Rx,-\Rx) -- (\Rx,\Rx)
    node[pos=0.7, above left, rotate=45, ink] {$r=r_S$};
  \draw[horizon] (-\Rx,\Rx) -- (\Rx,-\Rx)
    node[pos=0.7, above right, rotate=-45, ink] {$r=r_S$};

  \draw[sing]
    plot[samples=250, domain=-\Rx:\Rx]
      (\x, {sqrt(\x*\x + \b*\b)});
  \draw[sing]
    plot[samples=250, domain=-\Rx:\Rx]
      (\x, {-sqrt(\x*\x + \b*\b)});

  \node[ink] at (0, {sqrt(\b*\b)+0.25}) {$r=0$};
  \node[ink] at (0, {-sqrt(\b*\b)-0.35}) {$r=0$};

  \foreach \k in {1,...,\nInterior}{
    \pgfmathsetmacro{\c}{(\k)*(\b*\b)/(\nInterior+1)}
    \draw[curve]
      plot[samples=220, domain=-\Rx:\Rx]
        (\x, {sqrt(\x*\x + \c)});
    \draw[curve]
      plot[samples=220, domain=-\Rx:\Rx]
        (\x, {-sqrt(\x*\x + \c)});
  }

  \foreach \k in {1,...,\nExterior}{
    \pgfmathsetmacro{\d}{2.5 + 1.8*\k}
    \draw[curve]
      plot[samples=220, domain=-\Ty:\Ty]
        ({sqrt(\x*\x + \d)}, \x);
    \draw[curve]
      plot[samples=220, domain=-\Ty:\Ty]
        ({-sqrt(\x*\x + \d)}, \x);
  }

  \node[ink, rotate=30]  at (0.8, 1.8) {$r< r_S$};
  \node[ink, rotate=-30] at (1.5,-2.25) {$r< r_S$};
  \node[ink, rotate=80]  at (3.,0.7) {$r>r_S$};
  \node[ink, rotate=80]  at (-3.,-0.7) {$r>r_S$};

\end{scope}

\draw[axis] (0,-\Ty) -- (0,\Ty);
\draw[axis] (-\Rx,0) -- (\Rx,0);

\node[above] at (0,\Ty) {$T$};
\node[right] at (\Rx,0) {$R$};

\end{tikzpicture}
\caption{The Kruskal diagram of the maximally extended Schwarzschild spacetime. The black hole is the region $\{r < r_S, T > 0\}$ and the white hole is the region $\{r < r_S, T < 0\}$. The wavy lines $r=0$ are the black and white hole singularities that are not part of the spacetime.}
\end{figure}
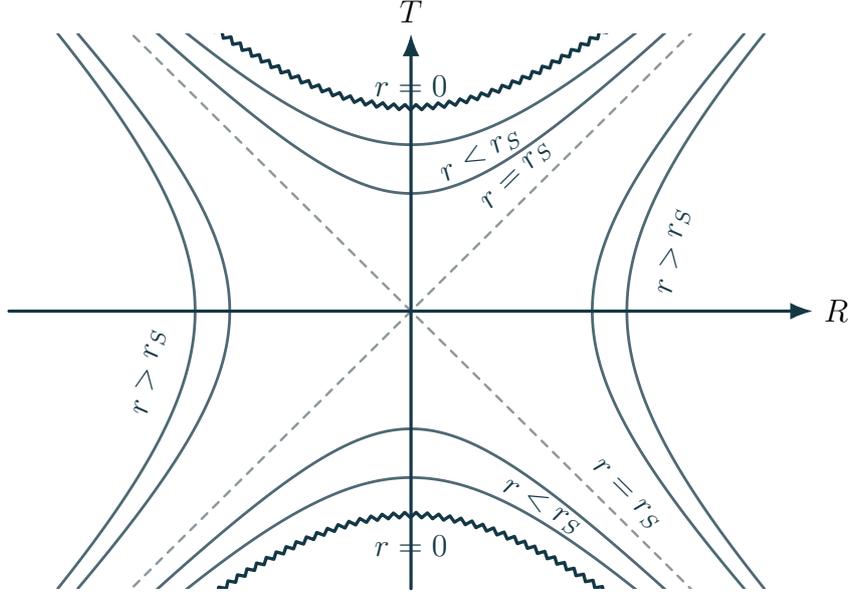

\subsection{Universe with a Big Bounce}
Consider a Friedman--Lema\^itre--Robertson--Walker (FLRW) metric of the form
\begin{equation}\label{eq:big_bounce}
    g = -dt^2 + \cosh^2(H t) g_\rmi{E}, \quad t \in (-\infty,\infty),
\end{equation}
where $g_E = \delta\ijd dx^i dx^j$ is the Euclidean metric on $\R^n$, and $H > 0$. The FLRW metrics, in general, are used in relativistic cosmology to model the evolution of the universe on large scales. The metric~\eqref{eq:big_bounce}, in particular, describes a `Big Bounce' cosmology: a universe contracting first to a minimum radius and then bouncing into an expanding phase. The spacetime $(\R^{1+n},g)$ is globally hyperbolic by \cite[Rmk. 12.27, Cor. 14.54]{oneill1983}. 

Now choose a cylinder $M = \{(t,x) \in \R^{1+n}: (\sum_{i=1}^n (x^i)^2)^{1/2} \leq R \}$ with $R > \pi/ H$. The boundary $\pt M$ is timelike. Consider then a maximal null geodesic $\gamma: [0,l] \to M$ with $l \in [0,\infty]$, $\gamma(0) \in \pt M$, and $\dot \gamma (0)$ pointing inside. To study light rays in $M$, it is convenient to introduce the conformal time coordinate
\begin{equation}
    \eta(t) = \int_{-\infty}^t \frac 1 {\cosh(H t')} dt' = \begin{cases}
        - H\pminus 1 \cot\pminus1(\sinh(H t)), &t < 0 \\
        H\pminus 1(\pi - \cot\pminus1(\sinh(H t))), &t \geq 0.
    \end{cases}
\end{equation}
The conformal time satisfies $\eta \in (0, \pi /H)$. On time slices we use the spherical coordinates centered around the initial point of a null geodesic. In the new coordinates, the metric reads
\begin{equation}
    g(\eta,x) = a(\eta)^2(-d\eta^2 + dr^2 + r^2 g_{S^{n-1}}),
\end{equation}
where $a(\eta) = \cosh(H t(\eta))$ and $g_{S^{n-1}}$ is the metric on the unit $(n-1)$-sphere. Since $g$ is conformally equivalent to the Minkowski metric $g_\rmi{M}$, null geodesics of $g$ are null pregeodesics of $g_\rmi{M}$, i.e., as point-sets they are straight line segments in $\R^{1+n}$. A null curve $\gamma(s) = (\eta(s), r(s),0,...,0)$ with $\eta(0)=\eta_0, r(0) = 0$ satisfies
\begin{equation}
    0 = a\pminus 2 g(\dot \gamma(s), \dot \gamma(s)) = -\dot \eta(s)^2 + \dot r(s)^2,  
\end{equation}
which implies that
\begin{equation}
    r(\eta) = \eta - \eta_0.
\end{equation}
The supremum for the coodinate distance $r$ traversed by a light ray is thus $\pi/H$. 
Consequently, $J^+ (U) \cap \pt M = \varnothing$, where $U$ is given by
\begin{equation}
    U = \{ (t,x) \in M: \big(\sum_i (x^i)^2\big)^{1/2} < \eta(t) + R - \pi/H \},
\end{equation}
that is, no signal sent from the central region reaches the boundary. On the other hand, $J^-(U') \cap \pt M = \varnothing$, where $U'$ is defined as 
\begin{equation}
    U' = \{ (t,x) \in M: \big(\sum_i (x^i)^2\big)^{1/2} < R -\eta(t) \}
\end{equation}
By Propositions~\ref{prop:nonuniqueness} and \ref{lem_data_implication} we again have a counterexample to Question~\ref{question:uniqueness}.

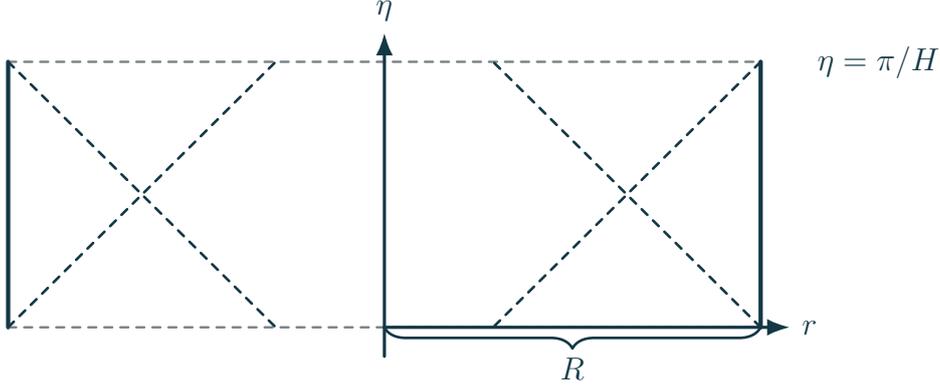
\begin{figure}
\begin{tikzpicture}[line cap=round, line join=round, scale=1.1]

\definecolor{ink}{RGB}{20,55,70}
\definecolor{soft}{RGB}{120,130,135}

\tikzset{
  wall/.style     ={ink, line width=1.6pt},
  dashedln/.style ={soft, dashed, line width=0.9pt},
  null/.style     ={ink, dashed, line width=1.0pt},
  axis/.style     ={ink, line width=1.2pt, -{Latex[length=3mm]}}
}

\def\L{9}
\def\H{3.2}
\def\Aext{0.35} 

\coordinate (BL) at (0,0);
\coordinate (BR) at (\L,0);
\coordinate (TL) at (0,\H);
\coordinate (TR) at (\L,\H);

\coordinate (O) at (\L/2,0);

\coordinate (B1) at (\H,0);
\coordinate (B2) at (\L-\H,0);

\coordinate (M1) at (\H,\H);
\coordinate (M2) at (\L-\H,\H);

\draw[wall] (BL) -- (TL);
\draw[wall] (BR) -- (TR);

\draw[dashedln] (TL) -- (TR);
\draw[dashedln] (BL) -- (BR);

\draw[null] (TL) -- (B1);
\draw[null] (TR) -- (B2);

\draw[null] (BL) -- (M1);
\draw[null] (BR) -- (M2);

\draw[decorate, decoration={brace, amplitude=8pt}, ink, line width=1.0pt]
  (\L,0) -- (\L/2,0)
  node[midway, below=7pt, ink] {$R$};

\draw[axis] ($(O)+(0,-\Aext)$) -- ($(O)+(0,\H+\Aext)$)
  node[above] {$\eta$};

\draw[axis] (O) -- ($(BR)+(\Aext,0)$)
  node[right] {$r$};

\node[ink, anchor=west] at (\L+0.55,\H) {$\eta=\pi/H$};

\end{tikzpicture}
\caption{The spacetime cylinder with a timelike boundary in the Big Bounce spacetime. The region $U$ is the upward-opening truncated cone, and $U'$ is the downward-opening one.}
\end{figure}





\section{Proofs}\label{section:proofs}

We now prove Proposition~\ref{prop:nonuniqueness}, starting with three lemmas. Then we turn to the proof of Proposition~\ref{lem_data_implication}.

\begin{lemma}\label{lem_open}
Let $(\mathcal M, g)$ be a time-oriented Lorentzian manifold and let $U \subset \mathcal M$ be open. Then $J_g^\pm(U) = I_g^\pm(U)$.
\end{lemma}
\begin{proof}
Let us prove the claim for the causal future; the past version is analogous. 
Let $q \in J_g^+(U)$. Then there is $p \in U$ with $p \le q$. As $U$ is open, there is $\tilde p \in U \cap I^-(p)$. Now $\tilde p \ll p \le q$, and \cite[Cor. 14.1]{oneill1983} yields $\tilde p \ll q$. Hence $q \in I_g^+(U)$.
\end{proof}

\begin{lemma}\label{lemma:causality}
Let $U\subset \mathcal M$ be open. Suppose that $g$ and $g'$ are time-oriented Lorentzian metrics on $\mathcal M$ and that the metrics and the time-orientations coincide in $\mathcal M\backslash U$. Then $J^\pm_g(U) = J^\pm_{g'}(U)$.
\end{lemma}
\begin{proof}
Let us prove the claim for the causal future; the past version is analogous. 
Let $q \in J^+_g(U)$. By Lemma \ref{lem_open} there is a curve $\gamma : [a, b] \to \mathcal M$ that is future-directed and timelike with respect to $g$, and $\gamma(a) \in U$ and $\gamma(b) = q$. Write 
    \begin{align*}
c = \sup \{ s \in (a,b): \gamma(s) \in U \}.
    \end{align*}
Then $\gamma\rvert_{[c,b]}$ is future-directed and timelike with respect to $g'$. 
By continuity, there is $0 < \delta \leq c-a$ such that $\gamma\rvert_{[c - \delta,b]}$ is future-directed and timelike with respect to $g'$. 
Due to the definition of $c$ as a supremum, and the openness of $U$, there is some $s \in [c - \delta, b]$ such that $\gamma(s) \in U$. Hence $q \in I^+_{g'}(U)$.
The opposite inclusion is analogous.
\end{proof}

\begin{lemma}\label{lemma:tau_cutoff}
    Let $(\mathcal M,g)$ be globally hyperbolic with a Cauchy temporal function $\tau$, and let $K \subset \mathcal M$ be compact. Then $J_g^+(K) \cap \{ \tau \leq \tau_0\}$ and $J_g^-(K) \cap \{ \tau \geq \tau_0\}$ are compact for any $\tau_0 \in \R$.
\end{lemma}
\begin{proof}
    Let us prove the claim for the causal future; the past version is analogous. 
    Denote the closure of causal future pointing tangent vectors on $K$ as
    \[
    \mathcal C = \{ (p,v) \in T\mathcal M: p \in K,\, g(v,v)\leq 0,\, v\,\, \textnormal{future pointing or $v = 0$}\}.
    \]
    Let now $q \in J_g^+(K)$. Then there exists $p \in K$ with $p \leq q$, and by global hyperbolicity and \cite[Prop. 14.19]{oneill1983} there is a causal geodesic from $p$ to $q$. This implies that $\exp(\mathcal C) = J_g^+(K)$. Introducing an auxiliary Riemannian metric $\zeta$ on $\mathcal M$, consider the unit sphere bundle
    \[
    S K_\zeta = \{ (p,v) \in T\mathcal M: p \in K, \, \zeta(v,v) = 1 \},
    \]
    which is compact for compact $K$. Write $\mathcal C_1 = \mathcal C \cap SK_\zeta$, which also is compact as $\mathcal C$ is closed. Let $\mathcal D\subset \mathcal C_1 \times [0,\infty)$ be the domain such that for each $(p,v) \in \mathcal C_1$ the geodesic starting from $(p,v)$ and defined on $\mathcal D \cap (\{(p,v)\}\times [0,\infty))$ is future inextendible. Let then the map $\Phi: \mathcal D \to \mathcal M$ be defined by $\Phi(p,v;s) = \exp_p sv$, and let $\lambda: \mathcal C_1 \to [0,\infty)$ be the function such that
    \[
    \tau(\Phi(p,v;\lambda(p,v))) = \tau_0.
    \]
    Since $\tau$ is a Cauchy temporal function and $\alpha: s \mapsto \Phi(x,v;s)$ is a causal curve for each $(p,v) \in \mathcal C_1$, we have $\pt_s \tau(\alpha(s))>0$. From the implicit function theorem it then follows that $\lambda$ is in particular continuous. Thus, there exists the maximum $L = \max_{(p,v)\in \mathcal C_1}\lambda(p,v)$ and $J_g^+(K)\cap\{\tau\leq \tau_0\} \subset \Phi(\mathcal C_1 \times [0,L])$, which is compact. The claim then follows from the fact that $J_g^+(K)\cap\{\tau\leq \tau_0\}$ is closed. 
\end{proof}
\begin{proof}[Proof of Proposition~\ref{prop:nonuniqueness}]
    We express $g$ in the form provided by Theorem~\ref{thm:bernal}.
    Let $U' \subset U$ be a coordinate chart with coordinates $(\tau,x_1,...,x_n)$, and let $h$ be a constant curvature Lorentzian metric in $U'$ given in local coordinates by 
    \[
    h = \frac{R^2}{\tau^2}(-d\tau^2 + \delta\ijd dx^i dx^j),
    \]
    where $R > 0$.\footnote{The metric $h$ is a local representation of de Sitter metric, and $R$ is the curvature radius of de Sitter space, which has constant curvature $1/R^2$.}
    Then let $\chi \in \Cinfo(\mathcal M)$ be such that $\supp \chi \subset U'$, $0 \leq \chi \leq 1$, and that $\{\chi = 1\}$ has nonempty interior. Now define a metric $g'$ in $\mathcal M$ by
    \begin{equation}\label{eq:gprime}
    g' = (1-\chi) g + \chi h',
    \end{equation}
    where $h'$ is an arbitrary smooth extension of $h$ into $\mathcal M$. As the matrices of the coordinate components of the metrics $g$ and $h$ are of the form 
    \[
    (g\munud) =
    \begin{pmatrix}
       g_{\tau\tau} & 0 \\
       0 & g\ijd
    \end{pmatrix}, 
    \quad (h\munud) = 
    \begin{pmatrix}
        h_{\tau\tau} & 0 \\
        0 & h\ijd
    \end{pmatrix},
    \]
    where $g_{\tau\tau} < 0, h_{\tau\tau}< 0$, and $(g\ijd)$ and $(h\ijd)$ are positive-definite, the convex combination~\eqref{eq:gprime} is indeed a Lorentzian metric. 
    
Let us show that $\tau$ is a Cauchy temporal function on $(\mathcal M, g')$.  The differential of $\tau$ is timelike with respect to $g'$. Indeed
    \begin{align}\label{dtau_timelike}
g'(d\tau,d\tau) = 1/((1-\chi)g_{\tau\tau} + \chi h_{\tau\tau}) < 0.
    \end{align}
We let $\tau_0 \in \R$ and show that the level surface $\Sigma_0=\{\tau = \tau_0\}$ is a Cauchy hypersurface on $(\mathcal M,g')$. It follows from \eqref{dtau_timelike} that
\begin{equation}\label{eq:tau_grows}
    \pt_s\tau(\gamma(s)) = d\tau(\dot \gamma(s)) > 0
\end{equation}
for any future-directed timelike curve, which implies that such a curve intersects $\Sigma_0$ at most once. Moreover, we can reparametrize $\gamma$ in terms of $\tau$ and write 
    \begin{align}\label{gamma_param}
\gamma(\tau) = (\tau,x(\tau)).
    \end{align}

Suppose, for contradiction, that there is an inextendible timelike curve $\gamma: (a,b)\to \mathcal M$ that does not meet $\Sigma_0$ at all. We may assume that $\gamma$ is future-directed  and parametrized as in \eqref{gamma_param}. The past-directed case is treated in a symmetric fashion. We can further assume that $p = \gamma(c) \in \supp \chi$ for some $c\in(a,b)$; otherwise, $\gamma$ would be an inextendible timelike geodesic of $g$ that never meets $\Sigma_0$, which would contradict the global hyperbolicity of $(\mathcal M,g)$.
Since $\gamma$ is parametried in terms of $\tau$, we have $b \le \tau_0$.

Now choose an open neighborhood $V$ of $p$ with compact closure and such that $\supp \chi \subset V$. The graph $\gamma([c,b))$ is contained in $J^+_{g'}(p) \subset J^+_{g'}(V)$ and, due to Lemma~\ref{lemma:causality}, also in $J^+_g(V) \subset J^+_g(\overline{V})$. But as $\overline{V}$ is compact, it follows from Lemma~\ref{lemma:tau_cutoff} that 
    \begin{align*}
K := J^+_g(\overline{V}) \cap \{\tau \leq \tau_0\}
    \end{align*}
is compact.
We write $g^{(\tau)}$ and $h^{(\tau)}$ for the time-dependent Riemannian metrics on $\Sigma$ given in local coordinates by $g^{(\tau)} = g\ijd dx^i dx^j$ and $h^{(\tau)} = h\ijd dx^i dx^j$.
The convex combination $G^{(\tau)} = (1-\chi)g^{(\tau)} + \chi h^{(\tau)}$ is a time-dependent Riemannian metric on $\Sigma$ as well. Due to compactness, we can find a Riemannian metric $H$ on $\Sigma$ such that $G^{(\tau)}(v,v) \geq H(v,v)$ for all $(\tau,x) \in K, v \in T_x \Sigma$.
    
Since $\gamma$ is timelike, we have
    \begin{align*}
0 > g'(\dot\gamma(\tau),\dot\gamma(\tau)) 
= (1-\chi)g_{\tau\tau} + \chi h_{\tau\tau} + G^{(\tau)}(\dot x(\tau), \dot x(\tau)).
    \end{align*}
As the functions $g_{\tau\tau}$ and $h_{\tau\tau}$ are negative, we get for $\tau \in [c, b)$
    \begin{equation}\label{eq:gamma_est}
H(\dot x(\tau), \dot x(\tau)) \le G^{(\tau)}(\dot x(\tau), \dot x(\tau)) < (1-\chi)\abs{g_{\tau\tau}} + \chi \abs{h_{\tau\tau}}. 
    \end{equation}
Since $\gamma\rvert_{[c,b)}$ takes values in the compact set $K$, we get from~\eqref{eq:gamma_est} that 
\[
H(\dot x(\tau),\dot x(\tau))^{1/2} \leq C
\]
for some $C>0$. Due to compactness of $K$, the Lipschitz bound implies that there is $x_0 \in \Sigma$ such that $\gamma(\tau)=(\tau,x(\tau)) \to (b,x_0) \in K$ as $\tau \to b$, which is a contradiction since $\gamma$ was inextendible. Hence, all inextendible timelike curves intersect $\Sigma_0$ exactly once.

We have shown that $\tau$ is a Cauchy temporal function on $(\mathcal M, g')$, which then is globally hyperbolic. Observe that $U \cap \p M = \varnothing$ in view of \eqref{unreachable}. Thus $\p M$ is timelike with respect to $g'$.

    Next, we prove that $g$ and $g'$ are not isometric for a suitable choice of $R > 0$. Let $S_g$ be the scalar curvature associated with $g$. Denote by $\textnormal{Cr}(S_g)$ the set of critical values of $S_g: M \to \R$. By the Morse--Sard theorem, $\textnormal{Cr}(S_g)$ has measure zero in $\R$; in particular, $\textnormal{Cr}(S_g) \neq \R$. Then we may choose the constant curvature metric $h$ so that $S_h = c \in \R\backslash \textnormal{Cr}(S_g)$. But $S_h$ being constant, $c$ is a critical value of $S_{g'}$ since the interior of $\{\chi =1\}$ is nonempty and $S_{g'}\rvert_{\{\chi = 1\}}=S_h$. Thus, $\textnormal{Cr}(S_g) \neq \textnormal{Cr}(S_{g'})$ and, consequently, there is no diffeomorphism $F: M \to M$ such that $F^*S_g = S_{g'}$. Hence, $g'$ is not isometric to $g$.

    Finally, let us prove that the source-to-solution maps of $g$ and $g'$ coincide. First we note that, by Lemma \ref{lemma:causality}, we have $J^\pm_g(U) = J^\pm_{g'}(U)$. We then simply refer to these sets as $J^\pm(U)$. We consider the problem
    \[
    \Box_g u = f \quad \textnormal{in } \mathcal M, \quad \supp u \subset J_g^+(\supp f),
    \]
    where $f \in \Cinfo(\mathcal M \backslash M)$. By \cite[Thm. 3.1.1, 3.3.1]{bar2007}, there exists a unique solution $u$ to this problem. We consider the two cases in \eqref{unreachable} separately.
    
Case 1. Suppose $J^-(U) \cap \pt M = \varnothing$. By continuity, this implies that $J^-(U) \cap (\mathcal M \backslash M) = \varnothing$, which in turn gives $J^-(U) \cap \supp f = \varnothing$. This is equivalent with $U \cap J_g^+(\supp f) = U \cap J_{g'}^+(\supp f) = \varnothing$. Thus, if $u$ solves $\Box_g u = f$, then it also solves $\Box_{g'} u = f$, and vice versa, since $g' = g$ in $\mathcal M \backslash U$. Thus, the source-to-solution maps are identical.

Case 2. Suppose $J^+(U) \cap \pt M = \varnothing$. Let $u$ and $u'$ be the solutions to 
        \[
        \Box_g u = f, \quad \Box_{g'}u' = f,
        \]
        respectively, and write $w = u - u'$. Then $\Box_g w = \Box_g u - \Box_g u' =: \phi$, and since $g = g'$ outside $\supp \chi \subset U$, we have 
        $$\phi\rvert_{\mathcal M\backslash \supp \chi} = \Box_g u - \Box_{g'}u' = f - f = 0.$$
        In particular, $\phi$ is compactly supported, and the unique solution to
        \[
        \Box_g w = \phi, \quad \supp w \subset J_g^+(\supp \phi)
        \]
        is supported in a subset of $J_g^+(U)$. Hence, $w = u - u' = 0$ in $\mathcal M \backslash M$, so the source-to-solution maps agree.
\end{proof}
\begin{remark}
    By the method of proof of Proposition~\ref{prop:nonuniqueness}, we see that $g'$ is not necessarily a small perturbation of $g$, in the sense that the curvature of $g'$ can deviate from that of $g$ by a large amount.
\end{remark}

To prove Proposition~\ref{lem_data_implication}, we prove first a lemma.

\begin{lemma}\label{lem_density}
Let $(\mathcal M, g)$ be a globally hyperbolic spacetime with a Cauchy temporal function $\tau$, and let $M \subset \mathcal M$ be a cylinder. Let $T \in \R$ and write 
\begin{align}\label{def_M_j}
M_1 = M \cap \{\tau \le T\},
\quad
M_2 = (\mathcal M \setminus M) \cap \{\tau \le T\},
\quad 
\Gamma = \p M \cap \{\tau \le T\}.
    \end{align}
Consider the map $A_g f = v|_{\Gamma}$, $f \in C_0^\infty(M_2)$, where 
    \begin{align}
\begin{cases}\label{eq_wave}
\Box_g v = f & \text{in $\mathcal M$}
\\
v|_{\tau \ll 0} = 0.
\end{cases}
    \end{align}
Then $A_g$ has dense range in $L^2(\Gamma)$.
\end{lemma}
\begin{proof}
Let $h \in L^2(\Gamma)$. It is enough to show that 
if $(h, A_g f)_{L^2(\Gamma)} = 0$ for all $f \in C_0^\infty(M_2)$, then $h = 0$. Consider the transmission problem 
    \begin{align}
\begin{cases}
\Box_g u_j = 0 & \text{in $M_j$, $j=1,2$,}
\\
u_1|_{\Gamma} - u_2|_{\Gamma} = 0
\\
\p_\nu u_1|_{\Gamma} - \p_\nu u_2|_{\Gamma} = h
\\
u_j|_{\tau \ll 0} = 0 & $j=1,2$.
\end{cases}
    \end{align}
This problem has a unique solution as can be seen by iterating \cite[Thm. 5.4]{williams1992} in local coordinates. Applying this theorem with $s > 0$ large (and therefore $r < 0$ large in absolute value), we see that $\p_\nu^m u_j$, $m = 0, 1$, are well-defined on $\Gamma$ for $j=1,2$. The theorem also guarantees enough  regularity for $u_j$ in the normal direction, so that the integration by parts below is justified.
Define $w = \sum_{j=1,2} 1_{M_j} u_j$. Then
    \begin{align}
(w, f)_{\mathcal D' \times C_0^\infty(M_1 \cup M_2)}
&= 
\sum_{j=1,2} (u_j, \Box_g v)_{\mathcal D' \times C_0^\infty(M_j)}
= 
(h, v)_{L^2(\Gamma)}
\\&= 
(h, Af)_{L^2(\Gamma)} 
= 
0.
    \end{align}
As $f \in C_0^\infty(M_2)$ is arbitrary, we obtain $u_2 = 0$. This again implies that 
    \begin{align}
\begin{cases}
\Box_g u_1 = 0 & \text{in $M_1$}
\\
u_1|_{\Gamma} = 0
\\
u_1|_{\tau \ll 0} = 0.
\end{cases}
    \end{align}
Hence $u_1 = 0$ and we obtain $h = 0$.
\end{proof}

\begin{proof}[Proof of Proposition \ref{lem_data_implication}]
We will apply Lemma \ref{lem_density} to $g$ and $\tau$. Write $v_g^f$ for the solution of \eqref{eq_wave} and $v_{g'}^f$ for the solution of the same system, but with $g$ replaced by $g'$.
Choose a sequence $f_j \in C_0^\infty(M_2)$ so that, writing $h_j = A_g f_j$, we have $h_j \to h$ in $L^2(\Gamma)$. 

It follows from $L_g = L_{g'}$ that $v_g^{f_j} = v_{g'}^{f_j}$ in $\mathcal M \setminus M$. In particular, $h_j = v_g^{f_j}|_{\p M} = v_{g'}^{f_j}|_{\p M}$, $\p_\nu v_g^{f_j}|_{\p M} = \p_\nu v_{g'}^{f_j}|_{\p M}$, and for $G = g, g'$
    \begin{align}
\begin{cases}
\Box_G v_G^{f_j} = 0 & \text{in $M$}
\\
v_G^{f_j}|_{\p M} = h_j 
\\
v_G^{f_j}|_{\tau \ll 0} = 0.
\end{cases}
    \end{align}
As $\tau$ is a temporal function for both the metrics, we get on $\Gamma$
    \begin{align*}
\Lambda_g h 
= \lim_{j \to \infty} \Lambda_g h_j 
= \lim_{j \to \infty} \p_\nu v_g^{f_j}
= \lim_{j \to \infty} \p_\nu v_{g'}^{f_j}
= \lim_{j \to \infty} \Lambda_{g'} h_j
= \Lambda_{g'} h,
    \end{align*}
where the convergence is justified for instance by \cite[Prop. 4.3]{2512.19601}, together with the closed graph theorem.
As $T > 0$ in \eqref{def_M_j} can be chosen arbitrarily, we conclude that $\Lambda_g h = \Lambda_{g'} h$ holds on the whole boundary $\p M$.
\end{proof}

\bibliographystyle{alpha}
\bibliography{master}

\end{document}